\documentclass[a4paper,reqno,10pt]{amsart}
\usepackage{amsthm,amssymb,amscd,amsfonts,amsmath,hyperref}
\newtheorem{theo}{Theorem}
\newtheorem{prop}{Proposition}
\newtheorem{lem}{Lemma}
\newtheorem{defi}{Definition}
\newtheorem{rem}{Remark}
\newtheorem{coro}{Corollary}

\numberwithin{equation}{section}
\begin{document}
\title{On Higgs bundles on Elliptic surfaces}
\author{Rohith Varma}
\address{Chennai Mathematical Institute, Plot H1, SIPCOT IT Park,
Siruseri, Kelambakkam, 603103, India.}
\email{rvarma@cmi.ac.in.}
\maketitle
\begin{abstract}
	Let $\pi : X \rightarrow C$ be a relatively minimal non-isotrivial elliptic surface over
        the field of complex numbers, where $g(C) \geq 2$. In this article,
	we demonstrate an equivalence between the category of semistable parabolic Higgs bundles on $C$, 
	and the category of semistable Higgs bundles on $X$ with vanishing second Chern class, and 
	determinant a vertical divisor.
\end{abstract}
\tableofcontents{}
\section{Introduction}
\subsection*{Motivation and statement of results}
Consider a relatively minimal elliptic surface 
$\pi : X \rightarrow C$ over $\mathbb{C}$,
with $\chi(X) > 0$. 
Let $c_1,\ldots,c_n$ be the set of points
on $C$ where the fibration $\pi$ has a multiple fiber of multiplicity
$m_i$ respectively. The 
data $(C,\textbf{c},\textbf{m}) := (C,c_1,\ldots,c_n,m_1,\ldots,m_n)$
defines a $2$-orbifold. It is well-known then that we have
a natural isomorphism of groups induced by $\pi$ (see \cite[Theorem 24, p~189]{fr98})
\[
\pi_1(X,*) \cong \pi_1^{orb}(C,*) \label{eq:fund} \tag{**}
\]
The orbifold fundamental group $\pi_1^{orb}(C,*)$ is defined as follows:
Recall the fundamental group $\pi_1(C - \{c_1,\ldots,c_n\},*)$
has $2g + n$ generators \\
$\alpha_1,\beta_1,\ldots,\alpha_g,\beta_g,\gamma_1,\ldots,\gamma_n$
subject to the relation 
\[[\alpha_1,\beta_1]\cdots [\alpha_g,\beta_g]\gamma_1 \cdots \gamma_n = 1.\] The orbifold
fundamental group $\pi_1^{orb}(C,*)$ is then defined to be
the quotient of $\pi_1(C - \{c_1,\ldots,c_n\},*)$ by the smallest
normal subgroup containing the elements $\gamma_i^{m_i}$.
Thus, $\pi_1^{orb}(C,*)$ is freely generated by the elements
$\alpha_1,\beta_1,\ldots,\alpha_g,\beta_g,\gamma_1,\ldots,\gamma_n$
subject to the relations 
\[
 [\alpha_1,\beta_1]\cdots[\alpha_g,\beta_g]\gamma_1\cdots \gamma_n = 1,
\,\ and \,\ \gamma_i^{m_i} = 1.\] \\
A natural class of elliptic surfaces with positive Euler characteristic
are the so called non-isotrivial elliptic fibrations. An elliptic surface
$\pi : X \rightarrow C$ is called isotrivial, if after passing to a 
finite cover $B \rightarrow C$, the surface $X \times_C B$ is birational
to a product $B \times E$, where $E$ is a complex elliptic curve.
An elliptic surface is called non-isotrivial if it is not isotrivial.\\ 

Now the space of Jordan equivalence classes of representations  
of $\pi_1(X,*)$ in $GL(n,\mathbb{C})$ can be identified with
the moduli space of $S$-equivalence classes of semistable
rank $n$ Higgs bundles on $X$, with vanishing Chern classes
from the work of Simpson (see \cite{simpI},\cite{simpII}).\\ 
Similarly, the space of Jordan
equivalence classes of representations of $\pi_1^{orb}(C,*)$,
correspond to $S$-equivalence classes of parabolic rank $n$
Higgs bundles on $C$. Here, by Jordan equivalence we mean 
the equivalence relation on the space of representations
given by identifying representations which have isomorphic
Jordan-Holder filtrations. The isomorphism \eqref{eq:fund} suggests
a natural correspondence between these moduli spaces.
If we restrict our representations to unitary representations,
then the corresponding
moduli spaces are that of $S$-equivalence classes of semistable
vector bundles on $X$ with vanishing Chern classes, and $S$-equivalence
classes of parabolic vector bundles on $C$ with parabolic degree $0$
(see \cite{chechu1,chechu2,chechu3,donaldson}).\\
An algebraic geometric isomorphism between these moduli spaces
was exhibited by Stefan Bauer in \cite{ba91} (see \cite{biswas1,biswas2,gantz} 
for related questions).\\
In this paper, our aim is to 
establish
a similar correspondence in the case of Higgs bundles on non-isotrivial
relatively minimal elliptic surfaces $\pi : X \rightarrow C$ with
$g(C) \geq 2$.\\
Recall, a Higgs bundle on a variety $Y$,
is a pair $(V,\theta)$ where $V$ is a vector bundle on $Y$ and 
$\theta \in Hom(V,V\otimes \Omega^1_Y)$, which satisfies
\[
\theta \wedge \theta = 0.
\]
Coming to our situation, fix a polarization $H$ on $X$ and consider the following categories:\\
$\mathcal{C}^{vHiggs}_X := \,\ $ The category of $H$-semistable Higgs 
bundles $(V,\theta)$ on X
with vanishing second Chern class and $det(V)$ a vertical divisor.\\
$\mathcal{C}^{ParHiggs}_{(C,\textbf{c},\textbf{m})} := \,\ $ The category of 
semistable parabolic Higgs bundles
on $C$ with parabolic structures above the points $c_i$ (with weights at $c_i$ belonging to $\frac {1}{m_i} \mathbb{Z} \cap [0,1)$).\\
Our main result is the following
\newtheorem*{thm:main}{Theorem \ref{thm:main}}
\begin{thm:main}
There is a natural equivalence of categories $\mathcal{C}^{vHiggs}_X$
and $\mathcal{C}^{ParHiggs}_{(C,\textbf{c},\textbf{m})}$.
\end{thm:main}

\subsection*{Strategy}
Our strategy, is the outcome of an attempt at adapting the 
study in \cite{ba91}, to the situation of Higgs bundles.\\ 
The category of
parabolic bundles (Higgs) on a curve $C$, with genus atleast $2$
is equivalent to the category of bundles (Higgs) on ramified
Galois covers, equivariant for the action of the Galois group.
Keeping this in mind, we consider an elliptic fibration
$\tilde{\pi} : \widetilde{X} \rightarrow \widetilde{C}$
with natural morphisms 
$q: \widetilde{X} \rightarrow X$ and 
$p: \widetilde{C}\rightarrow C$, such that
\[ \pi \circ q = p \circ \tilde{\pi}.
\]
We have further, $q : \widetilde{X} \rightarrow X$ is etale Galois.
We also have $p: \widetilde{C} \rightarrow C$ is Galois 
with the same Galois group as that of $q$.\\
Moreover, the fibration $\tilde{\pi}: \widetilde{X} \rightarrow \widetilde{C}$
is a non-isotrivial, relatively minimal elliptic surface with no
multiple fibers. We show that to prove Theorem \ref{thm:main}
it is enough to construct an equivalence of categories
of semistable equivariant (for the Galois group) Higgs bundles
on $\widetilde{X}$ and equivariant Higgs bundles on $\widetilde{C}$.\\ 
Subsequently, we argue that the equivariant situation as above can
be derived from Theorem~\ref{thm:main} applied to the fibration
$\tilde{\pi} :\widetilde{X} \rightarrow \widetilde{C}$.\\
So we are reduced to
proving Theorem \ref{thm:main} in the case of fibrations with
no multiple fibers.\\

Now in the case when the fibration has no multiple fibers, 
we show that every 
semi-stable Higgs bundle
$(V,\theta)$ on $X$ with vanishing second Chern class and
determinant vertical is the pull back of a semistable Higgs bundle
on $C$.
To that end, we first observe that it is enough to show that the bundle $V$ is the pull-back
of a bundle from the curve $C$. To see this let $W$
be a vector bundle on $C$ and consider the bundle $U := \pi^{*}(W)$
on $X$. We 
use Lemma \ref{lem:forms} and projection formula to
conclude
\[
H^{0}(X,\mathcal{E}nd(U,U)\otimes \Omega^1_X)) = 
H^{0}(C,\mathcal{E}nd(W,W)\otimes K_C).
\]
Hence, every Higgs field on the bundle $\pi^{*}(W)$ is the pull-back
of a Higgs field on $W$. \\
To show $V$ is the pull-back of a bundle on $C$,
it is enough to show its restriction to every fiber is trivial.
We reduce this further to showing that, the restriction of $V$ to
the generic fiber of $\pi$ is trivial. The generic fiber (possibly
after a base extension) of $\pi$ is an elliptic curve. Now we 
can use the beautiful classification results on vector bundles
over elliptic curves due to Atiyah (see \cite{atiyah}), to
study the generic restriction of $V$. This theme of understanding
the global picture by studying the generic fiber is something
which we use repeatedly in this article.

\subsection*{Related work and further comments}
The assumption $g(C) \geq 2$ has been used only to ensure the existence
of Galois covers with prescribed ramification points and ramification indices.
Hence, the results of section \ref{sec:2} and subsection \ref{subs:3} are 
valid without this assumption.
In particular, our proofs are valid in the case of a fibration with 
no multiple fibers without any assumption on the genus of $C$.\\

In \cite[section 5]{meta}, there is a discussion on the correspondence 
relating semistable Higgs bundles on elliptic surfaces with 
vanishing Chern classes and semistable
parabolic Higgs bundles (parabolic degree 0) on the curve. But the discussion 
is limited to
the case when the vector bundle underlying the Higgs bundle has a two step
Harder-Narasimhan filtration.\\ 
In \cite[Theorem 14.5]{garcia}, the authors 
consider the $1-1$ correspondence between the moduli space of 
semistable Higgs bundles
on $X$ with vanishing Chern classes and the moduli space of 
semistable parabolic Higgs bundles on the curve $C$ of parabolic degree $0$, 
arising as a consequence of the work of Simpson\cite{simpI,simpII}.
This correspondence is at the level of topological spaces.
The question of establishing an algebraic geometric correspondence
between these moduli spaces is proposed in \cite{garcia}, as a 
remark following the above mentioned theorem. We
are able to show such an isomorphism (see corollary \ref{coro:main})
in the case of non-isotrivial elliptic surfaces.\\

Further, as in \cite{ba91}
our study is not restricted to the situation of vanishing Chern classes.
We feel that results of this article must be true for isotrivial
elliptic surfaces with positive Euler characteristic as well, but 
we do not know how to prove it.

\subsection*{Acknowledgement}
The author wishes to thank his adviser Dr Vikraman Balaji for
his constant support and encouragement. He also would like
to thank Dr CS Seshadri for taking interest in this work.

\section {Preliminaries}
\label{sec:2}
All varieties considered are over the field of complex numbers unless mentioned
otherwise.
\subsection{Elliptic surfaces}
\label{subs:elliptic}
An Elliptic surface  is a fibered surface \\
$\pi : X \rightarrow C$
where the general fibers are genus $1$ curves and $X$, $C$ are 
a smooth projective surface and a smooth projective curve over $\mathbb{C}$
respectively. 
We call an elliptic surface as above relatively minimal
if there are no exceptional curves (curves with self intersection
number $-1$) on the fibers.\\
Just to be consistent with the definition of vertical
bundles (Definition~\ref{defi:vertical}) defined in the next subsection, we call a divisor 
$D$ \textit{vertical}, if $D$ is linearly equivalent to a divisor of the form 
$\Sigma_i r_i F_i$ where $r_i \in \mathbb{Q}$
and $F_i$ for every $i$ is a divisor corresponding
to a fiber at some point of $C$. We call a 
divisor $D$ \textit{vertically supported}
if $Supp(D)$ maps to a proper closed
subset of $C$ under $\pi$.
A \textit{vertically supported} divisor $D$ 
always satisfies $D^2 \leq 0$ and 
is
\textit{vertical} precisely when $D^2 = 0$.
For two divisors $D_1$ and $D_2$, we write $D_1 \equiv D_2$
if $D_1$ is numerically equivalent to $D_2$. 
The vertical divisors corresponding to various fibers
are all numerically equivalent. Hence as far as 
intersection theory is concerned, we may
work with a fixed fiber, say
we denote by $F$. The sheaf $R^1\pi_*(\mathcal{O}_X)$
is a line bundle on $C$ and we have in the case of 
relatively minimal elliptic surfaces
\[
\chi(X) = 12deg(L)
\]
where $L$ denotes the dual of the line bundle $R^1\pi_*(\mathcal{O}_X)$ on $C$.
We have the canonical bundle formula due to Kodaira (see \cite [Theorem 15,~p 176]{fr98})
\[
K_X \cong \pi^*(K_C \otimes L) \otimes \mathcal{O}_X(\Sigma_i(m_i-1)F_i).
\]
where $F_i$ are effective divisors on $X$ whose G.C.D of the coefficients
of the components are $1$ and the multiple fibers of $\pi$ are precisely 
of the form $m_iF_i$.\\
Hence, on a relatively minimal elliptic surface, we have
the canonical divisors of $X$ are vertical divisors.
If $Y \rightarrow C$ is an elliptic surface 
which is not relatively minimal, then assume after
blowing down the exceptional curves $\{E_1,\ldots,E_r\}$
we get a relatively minimal model say $X$. We then have
\[
K_Y \cong K_X \otimes \mathcal{O}_Y(E_1 + \ldots + E_r).
\]
Hence, $K_Y$ can be represented by vertically supported divisor. 
In particular, we have for any elliptic surface $X \rightarrow C$,
\[
K_X.F = 0.
\]
We will need the following
characterization of vertical divisors in the 
subsequent sections
\begin{lem}
Assume $\chi(X) > 0$. Then for a divisor $D$,
\[
D \,\ \text{vertical} \iff D.F = 0\,\ and\,\  D^2 = 0
\]
\label{lemm: vertical divisor}
\end{lem}

\begin{proof}
$D$ vertical clearly implies 
\[D.F = 0 = D^2.\]
Conversely, assume $D.F = 0$ and $D^2 = 0$.
Let $H$ be an ample divisor on $X$.
Choose a pair of integers $m,n$ such that 
$(mD + nF).H = 0$. Since now $(mD + nF)^2 = 0$,
we get from Hodge index theorem on surfaces
that $mD + nF \equiv 0$ and $D \equiv rF$
where $r \in \mathbb{Q}$.
If $D$ is vertically supported
 then $D = aF$, with
$a \in \mathbb{Q}$. 
This is the case if $D$ is effective.
Hence, 
to conclude the proof
it is enough to show $D_l = D + lF$ is effective
where $l \in \mathbb{N}$ as $D_l$ satisfies
the hypothesis of the Lemma and the preceding discussion
applied to $D_l$ says $D_l$ is vertical and hence
so do $D = D_l - lF$.
To see this choose $l>>0$,
so that $(D_l).H > (K_X).H$. Then
\[
H^{2}(X,\mathcal{O}_X(D_l)) = H^{0}(X, Hom(D_l,K_X))^{*} = 0.
\]
Applying Riemann-Roch theorem, we see that 
\[
H^{0}(X,\mathcal{O}_X(D_l)) = H^{1}(X,\mathcal{O}_X(D_l)) + \chi(\mathcal{O}_X) > 0 
\]
and hence we have $D_l$ is effective.
\end{proof}

The following Lemma is necessary for our study
\begin{lem}(see \cite{121880})
Let $X$ be a non-isotrivial elliptic surface with no multiple fibers. Then
the natural map 
\[
K_C \rightarrow \pi_*(\Omega^1_X)
\]
is an isomorphism.\label{lem:forms}
\end{lem}

\begin{proof}
Consider the short exact sequence
\[
0 \rightarrow \pi^*(K_C) \rightarrow \Omega^1_X \rightarrow \Omega^1_{X/C} \rightarrow 0 .
\]
Applying $\pi_*$, we get the following long exact sequence
\[
0 \rightarrow K_C \rightarrow \pi_*(\Omega^1_X) \rightarrow \pi_*(\Omega^1_{X/C}) 
\stackrel {\sigma}{\rightarrow} K_C \otimes R^{1}\pi_*(\mathcal{O}_X)
\]
The sheaf $\pi_*(\Omega^1_{X/C})$ is a rank $1$ sheaf on $C$, while 
$K_C \otimes R^{1}\pi_*(\mathcal{O}_X)$ is a rank $1$ locally free sheaf.
The map $\sigma$ restricted to the generic fiber is the kodaira-spencer map
which is non-zero if $X$ is assumed non-isotrivial. Hence the kernel
of $\sigma$ is precisely $\pi_*(\Omega^1_{X/C})_{Tor} = \pi_*((\Omega^1_{X/X})_{Tor})$. 
So we have
\[
0 \rightarrow K_C \rightarrow \pi_*(\Omega^1_X) \rightarrow \pi_*((\Omega^1_{X/C})_{Tor}) \rightarrow 0.
\]

Now from \cite[Proposition 1]{liu00} we get $H^{0}(C,\pi_*((\Omega^{1}_{X/C})_{Tor})) = 0$.
But as $\pi_*((\Omega^{1}_{X/C})_{Tor})$ is a torsion sheaf on $C$, it has to be the $0$-sheaf
since a non-zero torsion sheaf on a curve always has sections.
Hence we have 
\[
K_C \stackrel {\cong} {\rightarrow} \pi_*(\Omega^1_X).
\]
\end{proof}

\subsection{Vertical Bundles}
We keep the assumption that $\pi : X \rightarrow C$ is a 
\textit{non-isotrivial}
elliptic fibration. Denote by $K$, the function field $k(C)$ of $C$.
For a vector bundle $V$ on $X$, we denote by $V_K$ the bundle
on $X_K : = X \times_C spec(K)$ given by pull back of $V$ to $X_K$ through the natural
map $X_K \rightarrow X$. Similarly, for an extension $L/K$ of fields
we denote by $X_L : = X \times_C spec(L)$ and $V_L$ the bundle
on $X_L$ given by the pull back of $V$ to $X_L$ through the morphism
$X_L \rightarrow X$.\\
Let us recall the definition of \textit{Vertical Bundles} 
as defined in \cite[Definition 1.3,p~512]{ba91}
\begin{defi}
\label{defi:vertical}
A rank $n$ vector bundle $V$ on $X$ is called vertical, if $V$ has a filtration
\[
(0) = V_0 \subset V_1 \subset \ldots \subset V_n = V.
\]
by sub-bundles $V_i$, with $V_i/V_{i-1} \cong \mathcal{O}_X(D_i)$
where $D_i$ are vertical divisors.
\end{defi}
The main result of this subsection is the following proposition
which relates vertical bundles $V$ on $X$ and $V_K$.
\begin{prop}
Let $V$ be a vector bundle with $c_2(V) = 0$ and $D = det(V)$
is a vertical divisor. Then, $V$ is vertical if
and only if $V_K$ is semistable.\label{prop:vert}
\end{prop}

\begin{proof}
If $V$ is vertical, then clearly $V_K$ is semistable. Now for the converse,
let $\bar{K}/K$ be an algebraic closure of $K$. Consider the elliptic curve 
$X_{\bar{K}}$ and vector bundle $V_{\bar{K}}$ on $X_{\bar{K}}$. From assumption
we have $V_{\bar{K}}$ is semistable with trivial determinant. From Atiyah's
classification results on vector bundles on elliptic curves (see \cite{atiyah}),
we have $V_{\bar{K}} \cong \oplus_i L_i I_{m_i}$ where $L_i$ are degree $0$
line bundles on $X_{\bar{K}}$ and $I_m$ denotes the unique indecomposable
bundle on $X_{\bar{K}}$ of rank $m$ and trivial determinant.
Let $L/K$ denote a finite Galois extension so that for every index $i$
$L_i \in Pic^{0}(X_L)$. We then have 
a decomposition of $V_L$ as $\oplus_i L_iI_{m_i}$. Let 
$f:\tilde{C} \rightarrow C$ be the finite Galois cover
of $C$ corresponding to $L/K$. Choose a minimal
resolution $\tilde{X}$ of $X \times_C \tilde{C}$.
Since $X$ was non-isotrivial, the same holds true
for $\tilde{X}$ and hence $\chi(\tilde{X}) > 0$.
Denote by $\tilde{V}$, the pull back of $V$
to $\tilde{X}$. 
The bundle $\tilde{V}$ also satisfies $c_2(\tilde{V}) = 0$
and $\tilde{D} = det(\tilde{V})$ is a vertical divisor
with $\tilde{D}^2 = 0$.
We have $\tilde{V}_{L} = V_L$
and hence has a filtration by the line bundles $L_i$. We
can extend this filtration on $\tilde{V}_L$ to
a filtration by torsion free subsheaves on $\tilde{V}$,
\[
(0)= \tilde{V}_0 \subset \ldots \tilde{V}_{n-1} \subset \tilde{V}_n.
\]
such that $\tilde{V}_i/\tilde{V}_{i-1} \cong \mathcal{O}_X(D_i) \otimes I_{Z_i}$.
Using the additivity of Chern classes, we get
\begin{eqnarray}
\Sigma_i D_i & = & \tilde{D}, \label{e1}\\
\Sigma_{i<j} D_iD_j + lt(Z_i) &  = & 0.\label{e2}
\end{eqnarray}
Squaring equation(\ref{e1}) and using the fact that $\tilde{D}^2 = 0$, we get
\begin{equation}
\Sigma_{i<j} D_iD_j = - \frac{1}{2} \Sigma_i D_i^2.\label{e3}
\end{equation}
Substituting equation (\ref{e3}) in equation (\ref{e2}), we get
\begin{equation}
\Sigma_i lt(Z_i) = \frac {1}{2} \Sigma_i D_i^2.\label{e4}
\end{equation}
Since by assumption $D_iF = 0$, we have 
\[D_i^2 \leq 0, \forall i.\]
On the other hand we have \[lt(Z_i) \geq 0, \,\ \forall i.\] Hence from equation 
(\ref{e4}) we get the only possibility is \[lt(Z_i) = 0 \,\ and \,\ D_i^2 = 0.\]
Now from Lemma ~\ref{lemm: vertical divisor} we can conclude $D_i$
are vertical divisors for all $i$. In particular
\[
L_i \cong \mathcal{O}_{X_L} \,\ \forall i.
\]
Now consider the short exact sequence
\[
0 \rightarrow L_0 \cong \mathcal{O}_{X_L} \stackrel {t_0} {\rightarrow} V_L \rightarrow V_L/L_0 \rightarrow 0.
\]
Let $G := Gal(L/K)$ be the Galois group. We have $G$ acts on $V_L$  and hence
on the sections $H^{0}(X_L,V_L)$. Now replace $t_0$ by $Tr(t_0) = \Sigma_{g \in G} g(t_0)$
and we can assume $t_0$ is $G$-invariant and 
hence  $L_0$ is a $G$-invariant trivial sub-bundle of $V_L$. Hence
by Galois descent we have a section $s_0 : \mathcal{O}_{X_K} \rightarrow V_K$
with $t_0$ being the induced section of $V_L$ via base change and
$V_L/L_0 \cong (V_K/s_0(\mathcal{O}_{X_K}))_{L}$. Replacing now $V_L$
by $V_L/L_0$ and $t_0$ by $t_1 : L_1 \cong \mathcal{O}_{X_L} \rightarrow V_L/L_0$
and repeating the argument we see that $V_K$ has a filtration by sub-bundles 
with sub-quotients all trivial line bundles. Extend this filtration to
a filtration of $V$ and as in the case of $\tilde{V}$, we see that
$V$ is vertical. Thus we have proved the proposition.
\end{proof}
Let us recall now the definition of Higgs bundles on a projective variety $Y$.
A Higgs bundle on $Y$ is a pair $(V,\theta)$, where $V$ is a vector bundle on $Y$
and $\theta : V \rightarrow V \otimes \Omega^1_Y$ is a homomorphism with
\[
\theta \wedge \theta = 0.
\]
We fix a polarization $H$ on $Y$ and let $r$ be dim($Y$).
We say Higgs bundle $(V,\theta)$ on $Y$ is semistable if
for every subsheaf $W \subset V$ preserved by $\theta$
(i.e $\theta(W) \subset W \otimes \Omega^1_Y$), we have
\[
c_1(W).H^{r-1}/rank(W) \leq c_1(V).H^{r-1} /rank(V).
\]
Now consider the case when $Y$ is a surface.
For a vector bundle $V$ on $Y$, Denote by $\Delta(V)$ the number $(r-1)c_1(V)^2 - 2rc_2(V)$
which is called the \textit{Bogomolov Number} of $V$. If a vector bundle $V$
admits a Higgs field $\phi$ so that $(V,\phi)$ is a semistable Higgs bundle
on $Y$ (with respect to $H$), then we have the Bogomolov inequality
\[
\Delta(V) \leq 0.
\]
Further if $\Delta(V) = 0$, then the pair $(V,\phi)$ is semi-stable
with respect to any other polarization on $Y$ \cite[Theorem 1.3]{br06}. 
Hence now as a corollary of Proposition ~\ref{prop:vert} we have the following 
generalization of \cite[Lemma 1.4,~p 512]{ba91}.

\begin{coro}
\label{coro:vert}
If $(V,\theta)$ is a semistable Higgs bundle with $c_2(V) = 0$ and $D = det(V)$ vertical,
then $V$ is a vertical bundle.
\end{coro}

\begin{proof}
From Proposition \ref{prop:vert} it is enough to show $V_K$ is semistable. If $V_K$ is not semistable,
Then since $X_K$ is a genus $1$ curve, The H-N filtration of $V_K$ induces
a decomposition $V_K = \oplus_{i=1}^j W_i$ where $W_i$ is the destabilizing subsheaf of 
$V_K/W_{i-1}$ if we set $W_0 = (0)$. In particular each $W_i$ is semistable and 
\[
deg(W_0) > \ldots > deg(W_j).
\]
Now $(\Omega^1_X)_K$ is a rank $2$ vector bundle on $X_K$ which is an extension of 
$\mathcal{O}_{X_K}$ by itself. In particular $(\Omega^1_X)_K$ is semistable of
degree $0$ and so the bundles $W_i \otimes (\Omega^1_X)_K$ are semistable
with $deg(W_i \otimes (\Omega^1_X)_K) = 2deg(W_i)$ and 
$rk(W_i \otimes (\Omega^1_X)_K) = 2rk(W_i)$. 
We have then
\[
\mu(W_0) = deg(W_0)/rk(W_0) > \mu(W_i \otimes (\Omega^1_X)_K) = deg(W_i)/rk(W_i), \,\ i \geq 2.
\]
So 
\[
H^{0}(X_K,Hom(W_0,W_i\otimes (\Omega^1_X)_K)) = (0), \forall i \geq 2.
\]
Hence, we have $\theta_K(W_0) \subseteq W_0 \otimes (\Omega^1_X)_K$.
Now extend $W_0$ to a torsion free sub-sheaf $W \subset V$ with $V/W$
torsionfree as well. Since $\theta_K$ preserves $W_0$, the Higgs field
$\theta$ preserves the subsheaf $W$. On the other hand, as we have
$c_1(W).F = deg(W_0) > 0$, for a suitable $m>>0$ and the polarisation
$H+mF$, the slope $W$ exceeds that of $V$. But since $\Delta(V) = 0$,
this will contradict the stability of $(V,\theta)$ with respect to $H$. 
\end{proof}

Before we end this section we would like to address two natural
questions regarding vertical bundles. The first one is 
about when a sub-sheaf of a vertical bundle $V$ is itself vertical.
Clearly such a sub-sheaf $N \subset V$ satisfies $c_1(N).F = 0$.
We will see below [Lemma \ref{lem:subvert}] that this condition is 
in fact sufficient. The other question is specific to the
case when $\pi$ has no multiple fibers. In this situation
there is a natural class of vertical bundles, which are
the pull backs of bundles on $C$ to $X$. If $V$ is such
a bundle then we have $V_K = \mathcal{O}_{X_K}^{\oplus r}$
where $r = rk(V)$. Once again this condition turns
out to be sufficient [Lemma \ref{lem:pullback}]. 
\begin{lem}
Let $V$ be a vertical bundle and $N \subset V$ a sub-sheaf with torsion
free quotient $V/N$. Then, $N$ is vertical precisely when 
$c_1(N).F = 0$.\label{lem:subvert}
\end{lem}

\begin{proof}
Since $N$ and $V/N$ are torsion free, we have $N_K$ and $(V/N)_K$
are locally free on $X_K$. Further by assumption both
are of degree $0$. On the other hand as $V$ is vertical,
$V_K \cong \oplus_i I_{k-i}$, where $I_{k_i}$ is the
unique indecomposable bundle of rank $k_i$ and trivial
determinant. Hence both $N_K$ and $(V/N)_K$
also admit filtrations where the successive quotients are
trivial line bundles. Any such filtration on $N_K$ and 
$(V/N)_K$ can be extended to a filtration on $N$
and $V/N$ with successive quotients all of rank $1$
and of the form $\mathcal{O}_X(D_i) \otimes I_{Z_i}$
where $D_i$ is a vertically supported divisor and $Z_i$
is a closed set of points on $X$. But this filtration
is also a filtration on $V$. Now an argument involving
Chern classes as in the proof of Proposition \ref{prop:vert}
gives us $Z_i = \emptyset$ for every $i$ and $D_i$
are vertical divisors. Hence both $N$ and $V/N$
are vertical.

\end{proof}

\begin{lem}
Assume $\pi$ has no multiple fibers. Then a vertical bundle $V$
is isomorphic to $\pi^{*}(W)$ where $W$ is a bundle on $C$
if and only if $V_K = \mathcal{O}_{X_K}^{\oplus r}$ where $r = rk(V)$.
\label{lem:pullback}
\end{lem}

\begin{proof}
Let $V$ be a vertical bundle with $V_K = \mathcal{O}_{X_K}^{\oplus r}$.
Since by assumption $\pi$ has no multiple fibers, a line bundle
corresponding to a vertical divisor restricts to the trivial
line bundle on any fiber of $\pi$. Hence $V$ restricted to any fiber
is an iterated extension of trivial line bundles. In particular
if for $c \in C$, we denote by $X_c$ by the fiber (scheme theoretic)
of $\pi$ above $c$ and $V_c$ the restriction $V\mid_{X_c}$,
then as $h^{0}(X_c,\mathcal{O}_{X_c}) = 1$, we have
\[
h^{0}(X_c,V_c) \leq r.
\]
The equality occurs precisely at the points $c \in C$
where $V_c$ is the trivial rank $r$ bundle on $X_c$.
Now from semi-continuity principle the set 
$Z=\{ c \in C \mid h^{0}(X_c,V_c) = r \}$ is a non-empty
closed subset of $C$. But on the other hand we have
if $\zeta \in C$, the generic point of $C$, then
$h^{0}(X_{\zeta},V_{\zeta}) = h^{0}(X_K,V_K) = r$. Hence
$\zeta \in Z$ and thus $Z = C$. Thus $V$ restricts to
the trivial rank $r$ bundle on every fiber and consequently
$V \cong \pi^{*}(\pi_*(V))$.
\end{proof}

\section{Main Theorem}
\label{secs:3}
Let $\pi: X \rightarrow C$ denote a non-isotrivial elliptic
fibration possibly with multiple fibers and $\chi(X) > 0$.  
Fix a polarization $H$ on $X$.
Denote by $\mathcal{C}^{vHiggs}_X$, the category whose objects
are ($H$-)semistable Higgs bundles $(V,\theta)$ on $X$ with
$c_2(V) = 0$, $det(V)$ a vertical divisor, and morphisms
being Higgs bundle morphisms. Let $\textbf{c}: = \{c_1,\ldots,c_l\}$
be the points on $C$ where the fibers of $\pi$ are multiple.
Let the multiplicities of these fibers be $\textbf{m} := \{m_1,\ldots,m_l\}$
respectively. Recall the notion of a parabolic vector bundle
on $C$. A parabolic vector bundle on $C$ with a parabolic
structure at a point $c \in C$, consists of a vector
bundle $V$, together with a Flag 
\[ F^{\bullet}(V_c) := (0) \subset F^1(V_c) \subset F^2(V_c) \subset \ldots F^{r}(V_c) = V_c.\]
and weights $\alpha_i \in \mathbb{R}$ assigned to each subspace $F^{i}(V_c)$
such that \[0 < \alpha_1 <  \ldots < \alpha_r \leq 1.\]
To such a parabolic vector bundle $(V,F^{\bullet}(V_c),\{\alpha_i\})$
we can associate a real number called the parabolic degree given by
\[
Pardeg(V) := deg(V) + \Sigma_i \alpha_idim(F^{i}(V_c)/F^{i-1}(V_c)).
\]
In general, if there are parabolic structures on more than one point,
then the definition of parabolic degree has to be appropriately
modified.
There is a natural induced parabolic structure on every sub-bundle
of $V$ and we have an obvious notion of semistability (stability) using
the parabolic degree instead of the usual degree.
Now we also can define a parabolic Higgs bundle. Since 
in literature there are two different notions of a Higgs field,
we want to specify what we mean by a parabolic Higgs field.
For a parabolic vector bundle $(V,F^{\bullet}(V_c),\{\alpha_i\})$
as defined above, a parabolic Higgs field is a morphism
\[
\phi : V \rightarrow V \otimes K_C(c)
\]
such that we have
\[
\phi(F^i(V_c)) \subseteq F^{i-1}(V_c) \otimes K_C(c).
\]
Now we can define semistability for a parabolic Higgs bundle
$(V,F^{\bullet}(V_c),\phi)$ as in the usual case by
restricting the slope condition to sub-bundles preserved
by the Higgs field $\phi$. The definition of parabolic
Higgs bundles in the case of parabolic structures at more
than one point is the same as above except we have to
replace $K_C(c)$ by $K_C(c_1 + \ldots + c_l)$, where
$c_i$ are the parabolic points. Assume from now on
that $g(C) \geq 2$. Further, assume the weights associated with
the filtration $F^{\bullet}(V_{c_j})$ at $c_j$ all 
are rational and lie in $\frac {1}{m_j} \mathbb{Z} \cap [0,1]$.
Let $\mathcal{C}^{ParHiggs}_{(C,\textbf{c},\textbf{m})}$
denote the category of parabolic semistable Higgs bundles
with weights as described above. The following theorem
is the main result of this article
\begin{theo}
\label{thm:main}
There is a natural equivalence of categories $\mathcal{C}^{vHiggs}_X$
and $\mathcal{C}^{ParHiggs}_{(C,\textbf{c},\textbf{m})}$.
\end{theo}
Since we have assumed $g(C) \geq 2$, there exists a Galois cover
$p : \widetilde{C} \rightarrow C$ with Galois group denoted by $\Gamma$
and the local ramification groups above $c_j$ being the 
cyclic group $\frac {\mathbb{Z}} {m_j \mathbb{Z}}$ for every $j$.
Let $\mathcal{C}^{\Gamma-Higgs}_{\tilde{C}}$ denote the category of
$\Gamma$-equivariant Higgs bundles on $\tilde{C}$.
We then have a natural equivalence of categories
\[
p^{\Gamma}_{*} : \mathcal{C}^{\Gamma-Higgs}_{\widetilde{C}} 
\stackrel {\cong}{\rightarrow } 
\mathcal{C}^{ParHiggs}_{(C,\textbf{c},\textbf{m})}
\]
Now consider the 
commutative diagram 
\[
\begin{CD}
\widetilde{X} @>q>> X\\
@VV\tilde{\pi}V @VV \pi V\\
\widetilde{C} @>p>> C
\end{CD}
\]
where $\widetilde{X}$ is a minimal desingularization of $X \times_C \widetilde{C}$
(see \cite[Section~1.6,p~95-108]{fr94}).
We have $\tilde{\pi} : \widetilde{X} \rightarrow \widetilde{C}$
is a relatively minimal non-isotrivial elliptic surface with no multiple fibers.
Further, we have $q$ is an etale Galois cover with Galois group $\Gamma$.\\

Denote by $\mathcal{C}^{\Gamma-vHiggs}_{\widetilde{X}}$, the category 
of $\Gamma$-semistable $\Gamma$-equivariant 
Higgs bundles on $\widetilde{X}$ with $c_2 = 0$ and $c_1$ vertical.
From Galois descent, we have an equivalence of categories
\[
q^{*} : \mathcal{C}^{vHiggs}_X \stackrel {\cong} {\rightarrow} 
\mathcal{C}^{\Gamma-vHiggs}_{\widetilde{X}}.
\]
Hence to prove the theorem it suffices to construct a natural
equivalence between the categories $\mathcal{C}^{\Gamma-vHiggs}_{\widetilde{X}}$
and $\mathcal{C}^{\Gamma-Higgs}_C$. To that end we will first
prove the theorem in the case when the elliptic fibration
has no multiple fibers.
\subsection{Proof of Theorem~\ref{thm:main} in the case of no multiple fibers}
\label{subs:3}
Denote by $\mathcal{C}^{Higgs}_C$, the category of semistable Higgs bundles on $C$.
we have a natural map 
\[
d\pi : \pi^{*}(K_C) \rightarrow \Omega^1_X
\]
Now for a Higgs bundle $(W,\phi)$ on $C$, let $V = \pi^{*}(W)$. Then 
we denote by $d\pi(\phi) \in Hom(V,V \otimes \Omega^1_X)$ 
the composition $(Id_V \otimes d\pi) \circ (\pi^{*}(\phi))$.
Clearly $d\pi(\phi) : V \rightarrow V \otimes \Omega^1_X$ is a Higgs field
on $V$. Now we have the following Lemma
\begin{lem}
If $(W,\phi)$ is a semistable Higgs bundle on $C$, then for any
chosen polarisation on $X$, the Higgs bundle $(V,d\pi(\phi))$ is
semistable on $X$
\end{lem}
\begin{proof}
Since $\Delta(V) = 0$, it is enough to prove that there exists
a polarization with respect to which  $(V,d\pi(\phi))$ is semistable.
Assume the contrary and let $H$ be a polarization for which the 
pair $(V,d\pi(\phi))$ is unstable. Since the bundle $V_K$ is 
trivial and hence semistable, for any sub-sheaf of $N \subset V$,
we have $c_1(N).F \leq 0$. Hence, changing the polarization from
$H$ to $H+ mF$ for $m>>0$, either turns $(V,d\phi(\phi))$
into a semistable Higgs bundle in which case we are 
done or else the maximal destabilizing sub-sheaf $V_{max}$
satisfies  $c_1(V_{max}).F = 0$. But as $V_K$ is trivial
and $(V_{max})_K$ is a degree $0$ sub-bundle of $V_K$,
the only possibility is $(V_{max})_K$ is itself trivial.
Hence so do $(V/V_{max})_K$. Now from Lemma~\ref{lem:subvert}
and Lemma~\ref{lem:pullback} we have $V_{max} \cong \pi^{*}(\pi_*(V_{max}))$.
Hence, $\pi_*(V_{max})$ is a sub-bundle of $W$ of rank same as that of $V_{max}$.
Further, we have 
\[
\mu(\pi_*(V_{max})) = \frac {c_1(V_{max}).H/F.H}{rk(V_{max})} > \mu(W) = \frac {c_1(W).H/F.H}{rk(W)}.
\]
and $\pi_*(V_{max})$ is invariant under $\phi$, which contradicts semistability of $(W,\phi)$.
Hence, $(V,d\pi(\phi))$ is semistable for the polarization $H$.
\end{proof}
Thus, we have a well defined functor 
\[
\pi^{*} : \mathcal{C}^{Higgs}_C \rightarrow \mathcal{C}^{vHiggs}_X
\]
given by
\[
(W,\phi) \mapsto (\pi^{*}(W),d\pi(\phi)).
\]
From Lemma~\ref{lem:forms} we have the natural map $K_C \rightarrow \pi_*(\Omega^1_X)$
is an isomorphism. Hence, if $V = \pi^{*}(W)$ for $W$ a bundle on $C$, then
from projection formula every Higgs field $\theta$ on $V$ is of the form
$d \pi(\phi)$ for $\phi$ a Higgs field on $W$. Hence, the functor $\pi^{*}$
is full and  faithful. We will see below that  $\pi^{*}$
is essentially surjective as well and hence is an equivalence of categories,
which proves Theorem~\ref{thm:main} when $\pi$ has no multiple fibers.
\begin{rem}
The statement for line bundles (even without the assumption
of non-isotriviality) is a consequence of Hodge theory for
complex surfaces. Recall we have under the assumption of $\chi(X) > 0$,
\[
g(C) = h^{1,0} = dim_{\mathbb{C}}(H^{1}(X,\mathcal{O}_X)) = dim_{\mathbb{C}}(H^{0}(X,\Omega^1_X)) = h^{0,1}.
\]
On the other hand, the dimension of the subspace
$H^{0}(X,\pi^{*}(K_C)) \subseteq H^{0}(X,\Omega^{1}_X)$ is $g(C)$ as well.
Hence, we have the equality 
\[
H^{0}(X,\pi^{*}(K_C)) = H^{0}(X,\Omega^1_X).
\]
In particular, every $1$-form on $X$ is the pull back of a form on $C$.
Now a rank $1$ Higgs bundle of the form in the theorem above is
a pair $(L,\theta)$ where $L$ is isomorphic to a line bundle of the form $\mathcal{O}_X(D)$
with $D$ vertical (hence in the case of no multiple fibers, $D$
is the pull back of a divisor on $C$) and $\theta$ is a $1$-form.
So the statement holds true for rank $1$ Higgs bundles as in the
theorem.
\end{rem}
Since we have assumed $X$ to be non-isotrivial, we have from Lemma
~\ref{lem:forms}, $\pi_*(\Omega^1_X)$ is the line bundle $K_C$
on $C$. Hence from semicontinuity principle, we have
\[dim_K(H^{0}(X_K,(\Omega^1_X)_K))= 1. \]
Consider now the restriction of
the short exact sequence
\[
0 \rightarrow \pi^{*}(K_C) \rightarrow \Omega^1_X \rightarrow \Omega^1_{X/C} \rightarrow 0.
\]
to $X_K$. Since $(\pi^{*}(K_C))_K \cong \mathcal{O}_{X_K} \cong (\Omega^1_{X/C})_K$,
we see that $(\Omega^1_X)_K$ is an extension of $\mathcal{O}_{X_K}$ by
$\mathcal{O}_{X_K}$. Up to isomorphism, there are only $2$ such bundles on
$X_K$, the one being the trivial rank $2$ bundle and the other the 
indecomposable bundle $I_2$ (see \cite{atiyah}). Since we have seen already that 
$dim_K(H^{0}(X_K,(\Omega_{X})_K)) = 1$, the bundle $(\Omega^1_{X})_K$
cannot be the trivial bundle and hence it is isomorphic to $I_2$.
So the pair $(V_K,\phi_K)$ is a $I_2$-valued Higgs pair on $X_K$.
Such an $I_2$ valued Higgs pair is equivalent to a morphism
\[
 I_2^{*} \rightarrow End(V,V)
\]
such that fiber wise the image lands inside a family of commuting matrices.
The following Lemma about $I_2$-valued Higgs pairs is what we need
for our purposes
\begin{lem}\label{lem:generic}
Let $E$ be an elliptic curve over a field $k$, and $\phi : V \rightarrow V \otimes I_2$
be an $I_2$-valued Higgs field. We then have, for any section 
$\alpha \in H^{0}(E,End(I_2,\mathcal{O}_E))$, the induced element
$\beta = \alpha \circ \phi \in H^{0}(E,End(V,V))$ is Nilpotent.
\end{lem}

\begin{proof}
The bundle $I_2$ is an extension of $\mathcal{O}_E$ by $\mathcal{O}_E$
and hence we have a short exact sequence
\[
0\rightarrow \mathcal{O}_E \stackrel {s} {\rightarrow} I_2 \stackrel {t}{\rightarrow} \mathcal{O}_E 
\rightarrow 0.
\]
Further 
\[
H^{0}(E,I_2) = k<s>, \,\ H^{0}(E,End(I_2,\mathcal{O}_E)) = k<t>.
\]
In particular, for $a \in H^{0}(E,I_2)$ and $b \in H^{0}(E,End(I_2,\mathcal{O}_E))$, we always
have 
\[ba = 0 \in H^{0}(E,\mathcal{O}_E).\] 
We also have 
\[
I_2 \cong I_2^{*}
\]
Fix an isomorphism as above and then we have 
\[
H^{0}(E,I_2^{*}) = k<t^{*}> , \,\ H^{0}(E, End(I_2^{*},\mathcal{O}_E)) = k<s^{*}> .
\]
Consider now the morphism (which we denote by $\theta$ as well)
induced by the Higgs field
\[
\theta : I_2^{*} \rightarrow End(V,V).
\]
We have a trace map $Tr_V : End(V,V) \rightarrow \mathcal{O}_X$
and $Tr_V \circ \theta \in H^{0}(E,End(I_2^{*},\mathcal{O}_E))$. 
Let $Tr_V \circ \theta  = \lambda s^{*}$ and $\alpha = \gamma t^{*}$.
Then 
\[
Tr(\beta) = Tr_V \circ \theta \circ \alpha  = \lambda \gamma s^{*} t^{*} = 0.
\]
Let $L/k$ be a finite extension so that we have a decomposition
of $V_L$ as direct sum of generalized eigenspaces of $\beta$,
\[
V_L = \bigoplus _ {\delta_j \in L} V^{\delta_j}_L
\]
Since $\theta$ point wise lands in a family of commuting endomorphisms,
we have $V_L^{\delta_j}$ are preserved by $\theta$. Hence, we have induced maps
\[
\theta^{\delta_j}: I_2^{*} \rightarrow End(V^{\delta_j}_L,V^{\delta_j}_L).
\]
and $\beta^{\delta_j} = \theta^{\delta_j} \circ \alpha$. In particular
\[
\beta = \oplus \beta^{\delta_j}.
\]
Now as in the case of $V$, we get 
\[rank(V^{\delta_j}_L) \delta_j = Tr(\beta^{\delta_j}) = 0.\]
Hence, either $\beta = 0$ or all the eigenvalues are $0$ and hence $\beta$
is nilpotent.
\end{proof}
As a consequence of the above Lemma we have the following
\begin{lem}
\label{lem:main}
Let $(V,\theta)$ be an $I_2$-valued Higgs pair, with $V$ a semistable rank $r$
degree $0$-bundle on $E$. Then either , \\
(a)\,\ $V = L \otimes \mathcal{O}_E^{\oplus r}$ with $deg(L) = 0$, and 
$\theta \circ (id_V \otimes t) = 0$, or \\
(b)\,\ $\exists W \subset E$ with $deg(W) = 0$ and $\theta(W) \subset W \otimes I_2$.
\end{lem}

\begin{proof}
Consider the endomorphism 
\[
T = \theta \circ t : V \rightarrow V
\]
We have from Lemma~\ref{lem:generic}
that $T$ is a nilpotent endomorphism. Let $W : = Ker(T)$.
Now as $V$ is semistable of degree $0$, we have
$deg(W) \leq 0$. On the other hand by same reasoning $deg(Im(T)) \leq 0$.
Hence, $deg(W)$ is forced to be $0$. So if $\phi \neq 0$, then 
$W$ is a proper degree $0$ sub-bundle invariant under $\theta$
and we are done. If $\theta  = 0$, then $\theta$ factors
through $s : \mathcal{O}_E \rightarrow I_2$, i.e we have an
endomorphism $\phi : V \rightarrow V$ such that
\[
\theta  = (id_V \otimes s) \circ (\phi).
\]
Using Atiyah's classification results on bundles on elliptic curves \cite{atiyah},
it is easy to see that unless $V = L \otimes \mathcal{O}_{E}^{\oplus r}$
where $deg(L) = 0$, $\phi$ always leaves invariant a proper
degree $0$ sub-bundle of $V$.
\end{proof}
We have now all the ingredients to prove Theorem ~\ref{thm:main} when the fibration
has no multiple fibers. 
We provide below $2$ different arguments,
the first one though works only in the case when the 
Higgs bundle has no sub-Higgs sheaves.
\subsubsection{Higgs bundles with no sub-Higgs sheaves}
\label{subsec:3a}
\begin{proof}

Consider the spectral cover $Y \subset T^{*}X$ associated to a Higgs bundle
$(V,\theta)$. Let rank of $V$ be $r$.
The fact that $(V,\theta)$ has no sub-Higgs sheaves is equivalent
to $Y$ being irreducible and the natural map $q: Y \rightarrow X$ is
a finite map, which restricted to the smooth locus $Y^{sm}$
of $Y$ is a ramified $r$-sheeted cover of $q(Y^{sm})$.
Further, we have $V = q_*(L)$ where $L$ is a rank $1$
torsion free sheaf on $Y$. Now think of the Higgs field $\theta$
as a morphism
\[
\theta : TX \rightarrow End(V,V).
\]
For $x \in X$, the image of the induced morphism of vector spaces
\[
\theta(x) : T_xX \rightarrow End(V_x,V_x)
\]
by integrability condition on $\theta$ lies inside a commuting family of endomorphism.
Hence, the matrices in the image of $\theta(x)$ can be simultaneously triangularized
and the eigenvalues correspond to linear maps $T_x X \rightarrow \mathbb{C}$
or equivalently elements of $T^{*}_x X$ which is precisely the set $q^{-1}(x) \subset Y$.
Though there might not exist global sections of $T^*X$ which restrict to the
eigenvalues point wise, we can find sections of suitable symmetric powers of $T^*X$
which correspond to the co-efficients of the characteristic polynomials.
The discriminants of the point wise characteristic polynomials can also
be extended to a section of a suitable symmetric power of $T^*X$.
Let us  call it $\Delta(\theta)$. Now as we have seen already $T^*X$
restricts to the unique indecomposable rank $2$ bundle of trivial
determinant when restricted to the smooth fibers. Further it has
a unique section which if non-zero is nowhere vanishing. Assume 
now $x \in X$ with fiber of $\pi$ over $y = \pi(x)$ smooth and
$\Delta(\theta)(x) = 0$. Then $\Delta(\theta)$ vanishes on the entire
fiber $\pi^{-1}(y)$. In particular, as the vanishing locus of
$\Delta(\theta)$ is a closed set, it has to be nowhere vanishing
on an open set $\pi^{-1}(U)$ where $U \subset C$ is open.
In particular we see that $q$ is unramified on $q^{-1}(U)$
and the ramification locus is a vertically supported Divisor on $X$.
Denote the scheme theoretic fiber of $f$ over $X_K$ by $Y_K$
which is a disjoint union of elliptic curves over $K$. On the other hand
the torsion free sheaf $L$ restricts to a line bundle $L_K$ on $Y_K$
and $V_K = (q_K)_*(L_K)$. If we denote by $G$ the Galois group
(note here we do not assume $Y_K$ to be connected, but the Galois
group makes sense), then we have 
\[
q_K^{*}(V_K) = \oplus_{\sigma \in G}\sigma(L_K) 
\]
Hence $\#(G)(deg(L_K) = \#(G)deg(V_K) = 0$. On the other hand
$H^{0}(Y_K,L_K) = H^{0}(X_K,V_K) \neq 0$ and hence the only possibility
is $L_K \cong \mathcal{O}_{Y_K}$. But then since $q_K$ is unramified
$(q_K)_*(\mathcal{O}_{Y_K}) = \oplus_{i=1}^m(\oplus_{j=1}^{n_i} K_i^j)$
where $K_i$ are torsion line bundles on $X_K$ defining a connected subcover
$q_K^i : Y_K^{i} \subset Y_K \rightarrow X_K$. But as $V_K$ is already
an extension by trivial line bundles, the only possibility is $K_i = \mathcal{O}_{X_K}$
for every $i$ and hence $Y_K$ is a disjoint union of copies of $X_K$ and
$V_K = \oplus_ {j=1}^r \mathcal{O}_{X_K}$.
\end{proof}

\subsubsection{The general case}
\label{subsec:3b}

\begin{proof}
Consider the $I_2$ Higgs pair $(V_K,\theta_K)$ on $X_K$. We have $V_K$ 
is an iterated extension of trivial line bundles. Now from Lemma
~\ref{lem:main} we have either $V_K$ is trivial or has a degree $0$
(hence semistable) sub-bundle $W_K \subset V_K$ preserved by $\theta_K$.
Clearly, $W_K$ is also an iterated extension by trivial line bundles
as $V_K$ is so. We can extend $W_K$ to a sub-sheaf $W$ of $V$
with torsion free quotient $V/W$ and $\theta(W) \subseteq W \otimes \Omega^1_X$.
Further $det(W).F = 0$. Every sub-sheaf $Q \subset V$ preserved by $\theta$
satisfies $det(Q).F \leq 0$ as $V_K$ is semistable of degree $0$.
Now changing polarization from $H$ to $H+mF$ for a suitable $m \in \mathbb{N}$,
we can assume the subsheaf $V_{max} \subset V$, which has maximum slope
among the sub-sheaves preserved by $\theta$ satisfies $det(V_{max}).F = 0$.
In particular $(V_{max})_K$ is also an iterated extension by trivial line bundles
and so do the quotient $(V/V_{max})_K$. Chose a filtration by trivial
line bundles on $(V_{max})_K$ and $(V/V_{max})_K$ and extend them to $X$
as filtration on $V_{max}$ and $(V/V_{max})$ where the sub-quotients
are rank $1$ torsion free sheaves of type $\mathcal{O}_X(D_i)\otimes I_{Z_i}$
with $D_i$ being vertically supported divisors for every index $i$. Now observe this
filtration inturn gives a filtration on $V$ and as in the proof of
Proposition~\ref{prop:vert} we can see that infact $Z_i = \emptyset$
and $D_i$ are vertical divisors. Hence both $V_{max}$ and $V/V_{max}$
are vertical bundles. 
Denote the induced Higgs fields on $V_{max}$ and $V/V_{max}$ by
$\theta_0$ and $\theta_1$ respectively.
From the assumption both of them are semistable
Higgs bundles on $X$ as well of rank smaller than that of $V$.
Hence by induction we have semistable Higgs bundles $(W_0,\phi_0)$
and $(W_1,\phi_1)$ on $C$ such that
\[
(V_{max},\theta_0) \cong (\pi^{*}(W_0),\pi^{*}(\phi_0)), \,\ (V/V_{max},\theta_1) \cong (\pi^{*}(W_1),\pi^{*}(\phi_1)).
\]
Note that we have
\[
deg(W_0) = det(V_{max}).H/F.H  \leq deg(W_1) = det(V/V_{max}).H/F.H.
\]
Now consider the short exact sequence (infact a short exact sequence of Higgs bundles on $X$)
\begin{equation}\label{e:6}
0 \rightarrow V_{max} \rightarrow V \rightarrow V/V_{max} \rightarrow 0.
\end{equation}
Applying $\pi_*$ to equation~\ref{e:6}, we get a long exact sequence
\[
0 \rightarrow W_0 \rightarrow \pi_*(V) \rightarrow W_1 \stackrel {\eta}{\rightarrow} W_1 \otimes L^{-1} 
\]
where $L = R^1\pi_*(\mathcal{O}_X)^{-1}$. Now recall since $X$ is relatively minimal
and $\chi(X) > 0$, we have $deg(L) > 0$. The map $\eta$ is compatible
with the Higgs fields $\phi_0$ and $\phi_1$ on $W_0$ and $W_1$ respectively.
But \[ deg(W_0) > deg(W_1 \otimes L^{-1}) \]
and hence as $(W_0,\phi_0)$ and $(W_1,\phi_1)$ are semistable as Higgs bundles
on $C$, the morphism $\eta = 0$. Hence 
\[
rk(\pi_*(V)) = rk(V) \implies V_K \cong \mathcal{O}_{X_K}^{\oplus rk(V)}.
\]
\end{proof}

\subsection{Proof of Theorem ~\ref{thm:main} in the case of multiple fibers}
Recall the diagram
\[
\begin{CD}
\widetilde{X} @>q>> X\\
@VV\tilde{\pi}V @VV \pi V\\
\widetilde{C} @>p>> C
\end{CD}
\]
where $\tilde{\pi} : \widetilde{X} \rightarrow \widetilde{C}$ is a 
non-isotrivial 
relatively minimal elliptic surface with no multiple fibers.
We have $\widetilde{C} \rightarrow C$ is Galois with Galois
group $\Gamma$. Further $\widetilde{X} \rightarrow X$ is etale Galois with
Galois group also $\Gamma$.
From the previous subsection, we have an equivalence of categories
\[
\tilde{\pi}^{*} : \mathcal{C}^{Higgs}_{\widetilde{C}} \rightarrow 
\mathcal{C}^{vHiggs}_{\widetilde{X}}.
\]
Since $\tilde{\pi}$ is $\Gamma$ equivariant, the functor $\tilde{\pi}^*$
induces a functor
\[
\tilde{\pi}_{\Gamma}^{*} : \mathcal{C}_{\widetilde{C}}^{\Gamma-Higgs} \rightarrow 
\mathcal{C}_{\widetilde{X}}^{\Gamma-vHiggs}
\] 
We claim now the functor $\tilde{\pi}_{\Gamma}^*$ is an equivalence of categories.
To that end note that every $\Gamma$-semistable Higgs bundle on $\widetilde{X}$
is  semistable in the usual sense. Hence every 
object $(V,\theta) \in Ob(\mathcal{C}^{\Gamma-Higgs}_{\widetilde{X}})$
is isomorphic to a Higgs bundle of the form 
$(\tilde{\pi}^{*}(W),d\tilde{\pi}(\phi))$
as $\tilde{\pi}^*$ is an equivalence of categories. The only thing to 
verify is if
this isomorphism can be obtained in the category of $\Gamma$-equivariant 
Higgs bundles
on $\widetilde{X}$. 
The idea is to show that for any chosen isomorphism of Higgs bundles
$\lambda$ in $Isom((\tilde{\pi}^{*}(W),d\tilde{\pi}(\phi)),(V,\theta))$ 
we can provide $(W,\phi)$ with a natural $\Gamma$ structure such that
the isomorphism $\lambda$ is $\Gamma$ invariant.\\
To see this denote 
by $\tau^{\widetilde{X}}_g$ and $\tau^{\widetilde{C}}_g$
the respective automorphisms of $\widetilde{X}$ and 
$\widetilde{C}$ corresponding to $g \in \Gamma$.
Recall from the definition of $\Gamma$-equivariance, we have isomorphisms
\[
 \alpha_g \in 
Isom((V,\theta),((\tau^{\widetilde{X}}_g)^*(V),(\tau^{\widetilde{X}}_g)^{*}(\theta))).
\] 
satisfying
\[
(\tau^{\widetilde{X}}_h)^*(\alpha_g)\alpha_h = \alpha_{gh}.
\]
Now we have
\[
\tilde{\pi} \circ \tau^{\widetilde{X}}_g = 
\tau^{\widetilde{C}}_g \circ \tilde{\pi}, \,\ \forall g \in \Gamma. 
\]
Fix an isomorphism 
\[
\lambda \in Isom((\tilde{\pi}^*(W),d\tilde{\pi}(\phi)), (V,\theta)) 
\]
Set
\[
\beta_g = (\tau^{\widetilde{X}}_g)^*(\lambda) \circ \alpha_g \circ \lambda. 
\]
We have then
\begin{align*}
\beta_g 
&\in Isom((\tilde{\pi}^*(W),d\tilde{\pi}(\phi)),(\tilde{\pi}^*(\tau^{\widetilde{C}}_g)^*W),d\tilde{\pi}(\tau^{\widetilde{C}}_g)^*\phi)))\\
&= Isom((W,\phi), ((\tau^{\widetilde{C}}_g)^*(W),(\tau^{\widetilde{C}}_g)^*(\phi))).
\end{align*}
and clearly
\[
(\tau^{\widetilde{C}}_h)^*(\beta_g)\beta_h = \beta_{gh}.
\]
Hence $\beta_g$ induce a $\Gamma$-structure on 
$(W,\phi)$ such that the isomorphisms $\lambda$
are $\Gamma$-equivariant.\\

Let $\Delta$ be a vertical divisor and $d = \frac {\Delta.H}{F.H}$. Denote by
${M}_X^{Higgs}(r,\Delta,0)$ the moduli space of S-equivalence classes of rank $r$ $H$-semistable Higgs bundles
on $X$ with vanishing second Chern class and determinant numerically equivalent to $\Delta$.\\
Denote by ${M}_{(C,\textbf{c},\textbf{m})}^{ParHiggs}(r,d)$ the moduli space of S-equivalence classes of parabolic
Higgs bundles on $C$ with parabolic structures above the points $c_i$ (and weights above $c_i$
lying in $\frac {\mathbb{Z}} {m_i\mathbb{Z}}$) and parabolic degree $d$. \\
We have the following Corollary of Theorem~\ref{thm:main}
\begin{coro}
The moduli spaces ${M}_X^{Higgs}(r,\Delta,0)$ and ${M}_{(C,\textbf{c},\textbf{m})}^{ParHiggs}(r,d)$
are isomorphic as algebraic varieties.
\label{coro:main}
\end{coro}

\begin{proof}
As in the proof of Theorem~\ref{thm:main} it is enough to consider the case of a fibration
with no multiple fibers. We denote by 
$\mathcal{M}_X^{Higgs}(r,\Delta,0)$ ($\mathcal{M}_C^{Higgs}(r,d)$) the moduli functors whose corresponding
coarse moduli spaces are $M_X^{Higgs}(r,\Delta,0)$ ($M_C^{Higgs}(r,d)$
respectively).\\
Recall these moduli functors are from category of finite-type schemes over $\mathbb{C}$ to 
category of sets. For a given finite type scheme $T$, 
$\mathcal{M}^{Higgs}_X(r,\Delta,0)(T)$ ( $\mathcal{M}^{Higgs}_C(r,d)$)
is the set of equivalence classes of flat families of semistable Higgs bundles on $X$ 
of rank $r$, vanishing second Chern class and determinant numerically equivalent to $\Delta$
(semistable Higgs bundles on $C$ of rank $r$ and degree $d$ respectively)
parametrised by $T$.
Recall a family parametrised by $T$ corresponding to $\mathcal{M}_X^{Higgs}(r,\Delta,0)$, is
a pair $(\mathcal{V},\psi)$ where $\mathcal{V}$ is a sheaf on $X \times T$, flat over $T$
and $\psi \in Hom(\mathcal{V},\mathcal{V}\otimes_{\mathcal{O}_{X\times T}} pr_1^*(\Omega^1_X))$
where $pr_1$ denotes the projection map from $X \times T$ to $X$. Further for every closed
point $t \in T$, we have for the natural closed embedding $t: X \hookrightarrow X \times T$,
The pair 
\[
(V_t,\psi_t) : = (t^*\mathcal{V},t^*\psi)
\] 
is an object in $\mathcal{C}_X^{vHiggs}$ with $det(V_t) \equiv \Delta$.
Let $\pi_T$ denote the morphism 
\[\pi_T:= \pi \times id_T : X \times T \rightarrow C \times T.\]
From Theorem~\ref{thm:main}, we get the pair 
$(\mathcal{W},\phi):=((\pi_T)_*(\mathcal{V}),(\pi_T)_*(\psi))$
is a flat family of objects in $\mathcal{C}^{Higgs}_{C}$ with $deg(\mathcal{W}_t) = d$ for every $t$.
Thus we get a Natural transformation of functors
\[
\pi_* : \mathcal{M}_X^{Higgs}(r,\Delta,0) \rightarrow \mathcal{M}_C^{Higgs}(r,d)
\]
Conversely starting from a flat family $(\mathcal{W},\phi)$ of objects in $\mathcal{C}^{Higgs}_{C}$, 
with $deg(\mathcal{W}_t) = d$ for every $t$, parametrised by $T$, we have 
$((\pi)_T^*(\mathcal{W}),(\pi_T)^*(\phi))$ is a flat family of objects in $\mathcal{C}^{vHiggs}_X$ parametrised by $T$
as considered above.\\
Thus we get a natural transformation 
\[
\pi^* : \mathcal{M}_C^{Higgs}(r,d) \rightarrow \mathcal{M}_X^{Higgs}(r,\Delta,0)
\]
The composition $\pi_* \circ \pi^*$ and $\pi^* \circ \pi_*$ are clearly the identity 
transformations of the corresponding functors.\\
Hence the moduli functors $\mathcal{M}_X^{Higgs}(r,\Delta,0)$ and $\mathcal{M}_C^{Higgs}(r,d)$
are naturally equivalent and so 
the corresponding coarse moduli spaces $M_X^{Higgs}(r,\Delta,0)$ and $M_C^{Higgs}(r,d)$
are isomorphic as varieties.
\end{proof}

\bibliographystyle{amsplain}
\bibliography{p1}

\end{document}